\documentclass[12pt]{amsart}
\usepackage[pdftex]{graphicx}
\usepackage{color}
\usepackage[T1]{fontenc}
\usepackage{lmodern}
\usepackage[english]{babel}

\usepackage[centering,margin=2.5cm]{geometry}
\usepackage{upgreek}
\usepackage{amsfonts}
\usepackage{verbatim}
\usepackage{etex}
\usepackage{amsmath}
\usepackage[all, cmtip]{xy}
\usepackage{tikz-cd}
\usepackage{amssymb}
\usepackage{amsthm}
\usepackage{extpfeil}
\usepackage{bbm}
\usepackage{paralist}
\usepackage[T1]{fontenc}
\usepackage[colorlinks,citecolor=magenta,pagebackref=true,urlcolor=magenta,
pdftex]{hyperref}

\usepackage{amscd}

\usepackage{nicefrac}
\usepackage{microtype}

\usepackage{mathrsfs}
\usepackage[font=small,labelfont=bf]{caption}
\usepackage[labelformat=simple]{subcaption}

\makeatletter
\newtheorem*{rep@theorem}{\rep@title}
\newcommand{\newreptheorem}[2]{
\newenvironment{rep#1}[1]{
 \def\rep@title{#2~\ref{##1}}
 \begin{rep@theorem}}
 {\end{rep@theorem}}}

\makeatother

\theoremstyle{plain}
\newtheorem*{thm*}{Theorem}
\newtheorem{thm}{Theorem}[section]

\newreptheorem{mthm}{Main Theorem}
\newreptheorem{mcor}{Corollary}
\newreptheorem{mlem}{Lemma}
\newreptheorem{thm}{Theorem}

\newtheorem{cor}[thm]{Corollary}
\newtheorem{lem}[thm]{Lemma}
\newtheorem{fcts}[thm]{Facts}
\newtheorem*{lem*}{Lemma}
\newtheorem{prp}[thm]{Proposition}

\theoremstyle{definition}

\newtheorem{ex}[thm]{Example}

\newtheorem{rem}[thm]{Remark}

\newtheorem{problem}[thm]{Problem}

\newcommand{\wh}{\widehat}
\newcommand{\wt}{\widetilde}

\newcommand{\N}{\mathbb{N}}

\newcommand{\F}{\mc{F}}

\newcommand{\vanish}[1]{}

\def\wh{\widehat}
\def\wt{\widetilde}

\def\({\left(}
\def\){\right)}
\def\no={\,{\,|\!\!\!\!\!=\,\,}}
\def\wt{\widetilde}

\def\wh{\widehat}
\def\no={\,{\,|\!\!\!\!\!=\,\,}}

\newcommand{\xqedhere}[2]{%
  \rlap{\hbox to#1{\hfil\llap{\ensuremath{#2}}}}}

\newcommand\Defn[1]{\textbf{\color{blue}#1}}
\newcommand{\cm}[1]{}
\newcommand\mc[1]{\mathcal{#1}}

\newcommand\mr[1]{\mathrm{#1}}

\newcommand\scr[1]{\mathscr{#1}}

\newcommand{\K}{\Delta}

\newcommand\Sk{\mr{Sk}}
\newcommand\Cl{\mr{Cl}}

\DeclareMathOperator\cdr{{\upvarrho}}
\DeclareMathOperator\cdd{{\delta}}

\DeclareMathOperator{\lk}{\ell}

\DeclareMathOperator{\Lk}{lk}
\DeclareMathOperator{\eLk}{\wt{lk}}
\DeclareMathOperator{\St}{st}

\title[Higher Chordality]{Higher chordality: from graphs to complexes}
\author{Karim A.~Adiprasito}
\address{School of Mathematics, Institute for Advanced Study, Princeton, US
and \newline
Einstein Institute of Mathematics, University of Jerusalem, Jerusalem, Israel}
\email{adiprasito@math.fu-berlin.de, adiprasito@ihes.fr}
\author{Eran Nevo}
\address{Einstein Institute of Mathematics, The Hebrew University of Jerusalem, Jerusalem, Israel}
\email{nevo@math.huji.ac.il}
\author{Jose A.~Samper}
\address{Department of Mathematics, University of Washington at Seattle, USA}
\email{samper@math.washington.edu}
\date{\today}
\thanks{K.~A.~Adiprasito acknowledges support by an IPDE/EPDI postdoctoral fellowship, a Minerva postdoctoral fellowship of the Max Planck Society, and NSF Grant DMS 1128155.
}
\thanks{Research of E.~Nevo was partially supported by an ISF grant 805/11.
}
\begin{document}
\begin{abstract}
We generalize the fundamental graph-theoretic notion of chordality for
higher dimensional simplicial complexes by putting it into
a proper context within homology theory. We generalize some of the classical results of graph chordality to this generality, including the fundamental relation to the Leray property and chordality theorems of Dirac.
\end{abstract}

\maketitle
\section{Introduction}
Chordality is one of the most fundamental notions in graph theory. It finds application and stands in relation to graph colorings and perfect graphs, algorithmic graph theory, graph embeddings and Appolonian packings.

However, graph chordality only allows us to study the one-dimensional skeleton of a simplicial complex and therefore, for instance, only applies to the study of initial algebraic Betti numbers of the complex.

This motivates us to develop suitable notions of chordality for higher skeleta of simplicial complexes.
In this paper we develop a chordality notion in the realm of (simplicial) homology theory and study its fundamental properties.
In forthcoming papers, we will develop another chordality notion in the realm of stress spaces (from framework rigidity, e.g. \cite{Lee,TWW}, and McMullen's weight algebra), relate it to the homological notion developed here, and
explore its relevance in geometry (for polytopes and for simplicial nonpositive curvature) and in commutative algebra (for algebraic shifting and algebraic Betti numbers).

The classical definition states that a graph $G$ is chordal if for every simple cycle $z$ of length $\ge 4$ in $G$, there exists an edge $e\in G$ that connects two non-adjacent vertices of $z$.

One of the fundamental theorems concerning chordal graphs is the following 
result:
\begin{thm}
\label{thm:folkloreChord}
Let $G$ be a simple graph. Then the following are equivalent:
\begin{enumerate}
\item $G$ is chordal.
\item The complex of cliques of every connected induced subgraph of $G$ is contractible.
\item The complex of cliques of every connected induced subgraph of $G$ is acyclic, over any ring of coefficients.
\item The Stanley-Reisner ideal $I=(x_ix_j:\ ij\notin E(G),\ i,j\in V(G))$ has a linear resolution. Equivalently, $I$ has regularity $2$ (over some, equivalently any, field), unless $G$ is complete and $I=(0)$.
\end{enumerate}
\end{thm}

Parts of this result can be traced back at least to the sixties \cite{LekkerBoland}, and has been proven in a variety of settings, cf.\ \cite{Froberg88}, \cite{KalaiRig}, \cite{KM}, \cite{Green}.

Moving to higher dimensional simplicial complexes, attempts to generalize Fr\"{o}berg's result $(1)\Leftrightarrow (4)$ \cite{Froberg88} have been made in recent years.
Several combinatorial notions of chordality were introduced, e.g. \cite{Ha-VanT, Emtander, Woodroofe}, which imply the existence of a linear resolution over any field, but \emph{not} vice versa. Indeed, in Woodroofe's notion the Alexander dual $\K^{\vee}$ of the complex $\K$ is shellable. However, according to the Eagon-Reiner theorem \cite{Eagon-Reiner}, it is enough for $\K^{\vee}$ to be Cohen-Macaulay over the field to guarantee a linear resolution of the Stanley-Reisner ideal $I_{\K}$. In fact, already in dimension $2$, there are examples of non-shellable contractible complexes with connected vertex links, e.g. Hachimori \cite{Hachimori}, and thus are Cohen-Macaulay over \emph{any} field. Thus, for Hachimori's complex $\Gamma$, $\K=\Gamma^{\vee}$ is not chordal according to any of the above definitions but $I_{\K}$ has a linear resolution over any field (this answers \cite[Question 1]{Emtander}, see also \cite{Olteanu}).

In the other direction, Connon and Faridi found a necessary combinatorial condition for the existence of a linear resolution over all fields \cite{Connon-Faridi2}, and a
necessary and sufficient condition when the characteristic is $2$ \cite{Connon-Faridi}.
In the latter case, the homological and combinatorial descriptions are essentially the same.
Note that generally, having a linear resolution is characteristic dependent, as the $6$-vertex triangulation of the projective plane demonstrates.


The above discussion demonstrates that
the chordality notion is naturally best studied as a homological one.
A natural formulation for an extension would be to define the following. A simplicial complex $\Delta$ is \emph{decomposition $k$-chordal} if and only if every (simplicial) $k$-cycle $z$ in $\Delta$ (with respect to some coefficient ring) admits a \emph{decomposition}, i.e., it can be written as a sum of complete $k$-cycles, which support only vertices that are supported in $z$.
By a complete $k$-cycle we mean a cycle whose faces with nonzero coefficient are precisely the $k$-faces of a $(k+1)$-simplex.
\footnote{Decomposition $k$-chordal complexes slightly generalize the \emph{strongly triangulable} complexes of Cordovil, Lemos and Linhares Sales \cite{Cord-Lem-Sa}, in the sense that we allow facets of dimension $<k$. More importantly, our combinatorial analog of Dirac's elimination order is significantly more general than the one in \cite{Cord-Lem-Sa}, yet implies decomposition chordality, thus our Proposition \ref{prp:dirac_to_chordal} implies \cite[Thm.5.2]{Cord-Lem-Sa}.}

For decomposition chordality, the following theorem seems the appropriate generalization of Theorem \ref{thm:folkloreChord}. It sharpens the Eagon-Reiner theorem \cite{Eagon-Reiner} by providing a more ``economical" criterion for deciding whether the Alexander dual of a complex is Cohen-Macaulay.
\enlargethispage{5mm}

\newcommand{\RC}{\scr{Res}}
\newcommand{\DC}{\scr{Dec}}
\newcommand{\Res}{\mathrm{Res}}
\newcommand{\Dec}{\mathrm{Dec}}

\begin{thm}\label{mthm:persist}[Propagation of chordality]
Let $\Delta$ be any simplicial complex and assume that
\begin{compactenum}[\rm (1)]
\item $\Delta$ is decomposition $\ell$-chordal for all $k\le \ell\le {2k-1}$, and
\item $\Delta$ has no missing (a.k.a. empty) faces of dimension $> k$.
\end{compactenum}
Then
$\Delta$ is decomposition $\ell$-chordal for all $\ell \ge k$,
and is $k$-Leray.
\end{thm}

For $k=1$, this recovers the implication $(1)\Rightarrow(3)$ of Theorem \ref{thm:folkloreChord} for chordal graph theory. Analogs of the implications $(3)\Rightarrow(4)\Rightarrow(1)$ for complexes also hold and follow from known results; see Theorem \ref{thm:prp_exterior} below for details.
Theorem \ref{mthm:persist} is tight, see Proposition \ref{prop:tightPropagation}. In particular, decomposition $k$-chordality is not enough to guarantee higher decomposition chordality unless $k=1$ (or the trivial case $k=0$).
Part (2) of Theorem \ref{thm:folkloreChord} does not generalize under the assumptions of Theorem \ref{mthm:persist}, as the homotopy groups of $\Delta$ can be nontrivial in any dimension.

In Section \ref{sec:notation} we set notation, in Section \ref{sec:defRes} we
define \emph{resolution} chordality and relate it to decomposition chordality.
In Section \ref{sec:basics} we study fundamental properties of decomposition and resolution chordality with respect to basic operations on complexes, such as cone, join, link and
generalizations of Dirac's Gluing Lemma (cf.\ Lemma \ref{lem:gluing}).
In Section \ref{sec:Propagation} we prove the Propagation Theorem \ref{mthm:persist} and its consequences for regularity of the associated Stanley-Riesner ideal.
In Section \ref{sec:Dirac} we define a combinatorial analog of Dirac's elimination order for higher dimensions; this notion implies decomposition chordality. We also establish analogs of Dirac's Minimal Cut theorem for decomposition chordality.

\section{Basic notation.}\label{sec:notation}
Throughout, we allow any simplicial complex to be a \Defn{relative simplicial complex}, i.e.\ a pair of \Defn{abstract simplicial complexes} $\Psi=(\Delta, \Gamma)$, $\Gamma \subset \Delta$ (where an abstract simplicial complex is a downclosed subset of the powerset $2^{S}$ for some finite set $S$), cf. \cite{Stanley96}, \cite{AS}.
A $k$-dimensional simplicial complex is complete if it coincides with the $k$-skeleton of some simplex.

The \Defn{deletion} $\K-\sigma$ of a face $\sigma$ of $\K$ is the maximal subcomplex of $\K$ that does not contain $\sigma$; this is naturally extended to deletion of a collection of faces from $\K$.
For a subset $V$ of the vertices of $\K$, let $\K_{|V}$ denote the induced subcomplex of $\K$ on $V$.

Consider two simplicial complexes $\Delta_1$ and $\Delta_2$ whose vertex sets are disjoint. The \Defn{join} $\Delta_1*\Delta_2$ is the simplicial complex whose faces are of the form $F_1\cup F_2$ where $F_1 \in \Delta_1$ and $F_2 \in \Delta_2$.

Now, let $I\subset \{-1,0,1,2,\ldots\}$ denote any subset, and let $\K$ denote any simplicial complex. We then denote by $\K^{(I)}$ the collection of faces $\sigma \in \K$ with $\dim \sigma \in I$. As a special case, we obtain the
\Defn{$k$-skeleton} $\Sk_k \K:=\K^{(\le k)}$ of $\K$, and the \Defn{collection of $k$-faces} $\F_{k}\, := \, \K^{(k)}$.

A \Defn{$k$-clique} is a pure simplicial complex of dimension $k$ that contains all possible faces of dimension $\le k$ on its vertex set. With this, one can associate to any simplicial complex $\K$ its \Defn{k-cliques complex} (which contains $\K$), defined as \[\Cl_k \K\, :=\, \{\sigma\subset \F_0(\K): \Sk_k \sigma \subseteq \K \}.\]

Recall that the (closed) \Defn{star} and \Defn{link} of a face $\sigma$ in $\K$ are
the subcomplexes \[\St_\sigma \K:= \cup_{\sigma\subseteq\tau\in K}2^{\tau}\ \
\text{and}\ \ \Lk_\sigma \K:= \{\tau\setminus \sigma: \sigma\subset
\tau\in \K\}.\]
Notice that the previous notation is not standard, but it is convenient for the purpose of the paper. The \Defn{extended link} of a face $\sigma \in \K$ is the complex
\[\eLk_\sigma \K:=\St_\sigma \K-\sigma \cong \partial \sigma * \Lk_\sigma \Delta\]
where $\partial \sigma$ is the complex of proper faces of $\sigma$. Note that $\Lk_\sigma\subseteq \eLk_\sigma$ and if $\sigma=v$ is a vertex of $\K$, then $\Lk_v= \eLk_v$. In general, the dimension of $\eLk_\sigma \K$ is one less than the dimension of $\St_\sigma\K$, while the dimension of $\Lk _\sigma\K$ is the dimension of $\St_\sigma \K$ minus the cardinality of $\sigma$.

A \Defn{nonface} of $\K$ is, naturally, a simplex on $\K^{(0)}$ that is not a face of
$\K$. A \Defn{minimal nonface}, or \Defn{missing face}, of $\K$ is an
inclusion minimal nonface of $\K$. Equivalently, a simplex $\sigma$ is a
missing face of $\K$ if and only if $\partial \sigma\subset \K$, but $\sigma
\not\in \K$.

In simplicial homology, a total order on the vertex set $V$ is fixed.
We denote by $C_k(\Delta, R)$ the space of $k$-chains with coefficients in $R$, that is, the free $R$ module spanned by \emph{ordered} $k$-dimensional faces of $\Delta$. $C_{-1}(\Delta, R)$ is free of rank $1$, generated by the empty face. The vertex support of a $k$-chain $c$ is denoted by $c^{(0)}$, and we say that a face $\sigma\in\K$ is \Defn{supported} in $c$ if $\sigma$ is contained in a face with nonzero coefficient in $c$. Furthermore, for a $k$-chain $c$ and a subcomplex $\Gamma$ of $\K$, $c_{|\Gamma}$ denotes the restriction of $c$ to the summands supported in $\Gamma$.
The simplicial \Defn{boundary map} of $C_k(\Delta, R)$ is denoted by  $\partial$.
A $k$-chain $c$ with $\partial c=0$ is called a \Defn{$k$-cycle}. A $k$-cycle $z$ is \Defn{complete} if the faces with non-zero coefficients are the facets of the boundary of a $(k+1)$-simplex.

For ordered sets $\tau\subseteq\sigma\subseteq V$, let $\text{sgn}(\tau,\sigma)$ be $1$ or $-1$ depending on whether the permutation $\sigma\mapsto (\tau,\sigma - \tau)$ has an even or odd number of inversions, respectively. The boundary map $\partial$ applied to a face $\sigma$ viewed as a $k$-chain can be expressed as $\sum_{\sigma'} \text{sgn}(\sigma', \sigma) \sigma'$, where $\sigma'$ ranges over all $(k-1)$-dimensional faces of $\sigma$.

Notice that if the vertices of the join $\Delta_1*\Delta_2$ are ordered such that the vertices of $\Delta_1$ come before those of $\Delta_2$, then
the boundary operation on the join is given by
$\partial(\sigma_1*\sigma_2) = \partial(\sigma_1)* \sigma_2 + (-1)^{j_1+1}\sigma_1*\partial(\sigma_2)$, where $\sigma_1\in\Delta_1$ is $j$-dimensional and $\sigma_2\in\Delta_2$, and $\partial$ extends linearly to all chains.

We say that a simplicial complex is \Defn{$k$-Leray} if the reduced homology  $\tilde{H}_j(\Delta_{|V}) = 0$ for every $j\ge k$ and every subset $V$ of vertices of $\K$.

Let $W$ denote the vertex set of $\Delta$  and $\mathbb{F}$ a field. Consider the polynomial ring $S:=\mathbb{F}[x_w \, | \, w\in W]$. Let $I_{\K}$ be the ideal generated by the square free monomials indexed by the subsets of $W$ that are not faces of $\K$. The \Defn{Stanley Reisner ring} of $\K$ is the ring $\mathbb{F}[\K] := S/I_{\K}$. All resolutions of $\mathbb{F}[\K]$ considered in this paper are $\mathbb{Z}$-graded free resolutions of $S$-modules.

\section{Resolution and decomposition of cycles.}\label{sec:defRes}
Let $G$ denote a chordal graph, and let $\wt{G}=\Cl_1 G$ denote the complex induced by its $1$-cliques. Then, if $z$ is any $1$-cycle in $\wt{G}$, there exists a $2$-chain $c$ with $\partial c=z$ \emph{and} $c^{(0)}=z^{(0)}$. Equivalently, $z$ can be written as a sum of $1$-cycles of length $3$ that contain no vertices that are not already vertices of $z$.
This gives rise to two simple notions of higher chordality:

We say that a $(k+1)$-chain $c\in C_{k+1}(\K)$ is a \Defn{resolution} of a $k$-cycle $z\in Z_{k}(\K)$ if $c^{(0)}= z^{(0)}$ and $\partial c=z$.  We say that a (relative) complex in which every $k$-cycle admits a resolution is \Defn{$\cdr_k$-chordal}, or \Defn{resolution $k$-chordal}. We say that $\K$ is \Defn{$\cdd_k$-chordal}, or \Defn{decomposition $k$-chordal}, if every $k$-cycle $z$ can be written as a sum of complete $k$-cycles $(z_i)$ in such a way that $z_i^{(0)}\subset  z^{(0)}$ for all $z_i$. The $k$-th reduced homology group is denoted by $\tilde {H}_k(\Delta, R)$.

Let us observe some straightforward properties of resolution and decomposition chordality:
\begin{fcts}[Resolution and decomposition chordality]
\label{fact:Res-Dec-Chordality}
\mbox{}
\begin{compactenum}[\rm (1)]
\item A graph is $\cdd_1$-chordal if and only if it is chordal in the classical sense (regardless of the ring of coefficients).
\item A complex $\K$ is $\cdd_k$-chordal if and only if $\Cl_k \K$ is $\cdr_k$-chordal.
\item Any $\cdr_k$-chordal complex is $\cdd_k$-chordal.
\item A simplicial complex $\K$ is resolution $k$-chordal if and only if for every subset of vertices $V\subset \K^{(0)}$, we have in reduced homology $\tilde{H}_k(\K_{|V})\ =\ 0$. (Reduced homology is used to include the case when $k=0$.)
\item In particular, the simplex is resolution $k$-chordal for all $k$.
\item Resolution $k$-chordality only depends on the faces of dimension $k$ and $(k+1)$ of a simplicial complex, i.e., for two simplicial complexes with the same set of $k$- and $(k+1)$-faces, either both are $\cdr_k$-chordal or both are not $\cdr_k$-chordal.
\end{compactenum}
\end{fcts}

In higher dimension, we find it more useful to discuss resolution $k$-chordality than decomposition $k$-chordality, however, as we have seen, they are equivalent for $k$-clique complexes.

\section{Basic operations and chordality of complexes}\label{sec:basics}
An important observation is the following \Defn{cone lemma}. Note that on the level of simplicial chains, we have a natural map for any vertex $v\in\K^{(0)}$
\begin{align*}
\lk_v\ :\ C_k(\K)\ \ &\longrightarrow\ \ C_{k-1}(\Lk_v \K) \\
\sum_{\sigma \in \K^{(k)}} g_\sigma\cdot \sigma\ \ &\longmapsto\ \ \sum_{\sigma \in \K^{(k)}}
g_\sigma\cdot\text{sgn} (v,\sigma)\cdot \Lk_v \sigma = \sum_{\sigma \in \K^{(k)}:\ v\in\sigma}
g_\sigma\cdot\text{sgn} (v,\sigma)\cdot (\sigma-v)
\end{align*}
that descends to a quasiisomorphism of chain complexes
\begin{equation}\label{eq:cone_lem}
C_\bullet(\St_v \K, \St_v \K-v)\cong C_{\bullet -1}(\Lk_v \K).
\end{equation}
In particular, we have:
\begin{lem}[Cone Lemma]\label{lem:cone} A relative simplicial complex $(\St_v \K, \St_v \K-v)$ is resolution $k$-chordal if and only if $\Lk_v \K$ is resolution $(k-1)$-chordal.\qed
\end{lem}
On the other hand, as we shall see, this simple lemma is the main ingredient for any simplicial homology theory to satisfy the propagation of chordality property.

More generally, for any face $\tau\in\K$, the extended-link map
\begin{align*}
\lk_{\tau}\ :\ C_k(\K)\ \ &\longrightarrow\ \ C_{k-1}(\eLk_{\tau} \K) \\
\sum_{\sigma \in \K^{(k)}} g_\sigma\cdot \sigma\ \ &\longmapsto\ \ \sum_{\sigma \in \K^{(k)}}
g_\sigma\cdot \eLk_v \sigma = \sum_{\sigma \in \K^{(k)}:\ \tau\subseteq\sigma}
g_\sigma\cdot\text{sgn}(\tau,\sigma)\cdot (\partial\tau * (\sigma-\tau))
\end{align*}
descends to an isomorphism in homology
\[H_k(\St_{\tau} \K, \partial\St_{\tau}\K)\cong H_{k -1}(\eLk_{\tau} \K),\]
and thus,
\begin{lem}[Extended Cone Lemma]\label{lem:econe}
The relative complex $(\St_{\tau} \K, \partial\St_{\tau} \K)$ is resolution $k$-chordal if and only if $\eLk_{\tau} \K$ is resolution $(k-1)$-chordal.
\end{lem}
We conclude with the following result for the links:
\begin{lem}[Links Lemma]\label{lem:link}
Let $\K$ be any simplicial complex, and let $F$ denote any
$\ell$-dimensional face of $\K$. If $\K$ is resolution $k$-, $(k-1)$-, $\cdots$ $(k-\ell-1)$-chordal, then $\Lk_F \K$ is resolution $(k-\ell-1)$-chordal.
\end{lem}
\begin{proof}
Recall that if $F = \{v_1, \dots v_k)$, then $\Lk_F(\Delta)= \Lk_{v_1}(\Lk_{v_2}\dots \Lk_{v_k}(\Delta)\dots )$, and so it suffices to treat the case when $F$ is a vertex $v$.
Let $z$ denote any $(k-1)$-cycle in $\Lk_v \K$; consider a preimage $a=\lk_v^{-1} z$ supported in $\St_v \Delta$, and let $b$ denote a resolution of $z$ in $\K$ (which exists by the assumed $(k-1)$-chordality). Consider the $k$-cycle $\wh{z}:=a-b$. By $k$-chordality of $\K$, $\wh{z}$ admits a resolution $\wh{c}$ in $\K$; the $k$-chain $c:=\lk_v \wh{c}$ of $\Lk_v \Delta$ is the desired resolution of $z$.
\end{proof}
The chordality conditions in the above lemma are tight:
\begin{ex}\label{ex:LinksResChord}
Let $\Delta$ be the cone with apex $v$ over a square boundary $C$. Then $\Delta$ is resolution $2$-chordal but $C$ is not decomposition $1$-chordal.
\end{ex}

The next result characterizes the behavior of chordality under joins:

\begin{lem}[Join Lemma]\label{lem:susp}
Let $\Delta$ denote any simplicial complex, and let $c$ denote any $d$-cycle whose vertex support is disjoint from $\Delta^{(0)}$.
Let $\Delta\ast c$ denote the complex whose faces are $\sigma\cup\tau$ for $\sigma\in\Delta$ and $\tau$ supported in $c$.

If $\Delta\ast c$ is resolution $k$-chordal, then $\Delta$ is resolution $(k-d-1)$-chordal.
\end{lem}

\begin{proof}
Let $z$ denote any $(k-d-1)$-cycle in $\Delta$. Then $\wh{z}=z\ast
c$ is a $k$-cycle in $\Delta\ast c$, and therefore admits a
resolution $\wh{c}$. Note that $\wh{c}$ must have the form $\wh{c}=\tilde{c}*c$ and that $\tilde{c}$ resolves $z$.
\end{proof}

Let us finally establish a generalization of Dirac's Gluing Theorem \cite[Theorem 2]{Dirac}:
\begin{lem}[Dirac's Gluing Lemma]\label{lem:gluing}
Let $\K$, $\Gamma$ denote any two simplicial complexes and let $\mc{M} := \K\cup \Gamma$
\begin{compactenum}[\rm (i)]
\item If $\Delta$ and $\Gamma$ are resolution $k$-chordal complexes, and
$\Delta\cap\Gamma$ is resolution $(k-1)$-chordal, then $\mc{M}$ is
resolution $k$-chordal.

\item Similarly, if $\Delta$ and $(\Gamma, \Gamma\cap\Delta)$ are resolution $k$-chordal then so is $\mc{M}$.

\item If $\Delta$, $\mc{M}$ and $\Delta\cap \Gamma$ are resolution $k$-chordal, then $\Gamma$ is resolution $k$-chordal.
\end{compactenum}
\end{lem}

\begin{proof}
(i)
Let $z$ denote any $k$-cycle in $\mc{M}$, and consider the chains
$a:=z_{|\Delta}$ and $b=z-a$. Consider now $\wt{z}:=\partial a=-\partial
b\subset \Delta\cap\Gamma$. There exists a resolution $\wt{c}$
of $\wt{z}$ by $(k-1)$-chordality of $\Delta\cap\Gamma$, and there exist resolutions for
$a-\wt{c}$ and $b+\wt{c}$ by $k$-chordality of $\Delta$ resp.\ $\Gamma$.
The sum of the resulting resolutions gives the desired resolution of $z$.

(ii) Let $z,a,b$ be as before. As $(\Gamma, \Gamma\cap\Delta)$ is $\cdr_k$-chordal, it has a resolution $b'$ of $b$, so in $\mc{M}$ we have $\partial b'=b+c_b$ where the chain $c_b$ is supported in $\Delta\cap\Gamma$ and its vertex support is contained in the vertex support of $b$. Further, $\partial c_b=\partial a$. As $\Delta$ is $\cdr_k$-chordal,
it has a resolution $a'$ of $a-c_b$, so in $\mc{M}$ we get that $a'+b'$ resolves $z$.

(iii)
Now, for a $k$-cycle $z$ in $\Gamma$ with vertex set $V=z^{(0)}$
consider the Mayer--Vietoris sequence
\[
        \cdots\; \longrightarrow \,
        \wt{H}_{k}(\Delta_{|V}\cap \Gamma_{|V}) \, \longrightarrow \,
        \wt{H}_{k}(\Delta_{|V})\oplus \wt{H}_{k}(\Gamma_{|V}) \,
\longrightarrow\,
        \wt{H}_{k}(\Delta_{|V}\cup\Gamma_{|V})
        \, \longrightarrow \,
       \wt{H}_{k-1}(\Delta_{|V}\cap \Gamma_{|V}) \, \longrightarrow \;
        \cdots
    \]
By resolution $k$-chordality of $\Delta, \Delta\cup \Gamma$ and $\Delta\cap \Gamma$ we conclude $\wt{H}_{k}(\Gamma_{|V})=0$, proving that $\Gamma$ is resolution $k$-chordal as well.
\end{proof}

\section{Propagation of chordality}\label{sec:Propagation}
One of the fundamental theorems of chordal graph theory is the classical Theorem \ref{thm:folkloreChord}, asserting that for a chordal graph its $(1-)$clique complex is $1$-Leray, and equivalently, that its Stanley-Reisner ideal has a linear resolution.
For homological chordality, the following more detailed version of Theorem \ref{mthm:persist} seems the appropriate generalization:

\begin{thm}\label{thm:prp_exterior}
Let $\Delta$ denote any (abstract) simplicial complex without missing faces of dimension $> k$, and fix any field of coefficients. The following are equivalent:
\begin{compactenum}[\rm (1)]
\item $\Delta$ is resolution $i$-chordal for $i\in [k,2k-1]$.
\item $\Delta$ is resolution $i$-chordal for $i\ge k$.
\item $\Delta$ is $k$-Leray.
\item The Stanley--Reisner ring of $\Delta$ is of (Castelnuovo-Mumford) regularity at most $k$.
\end{compactenum}
If $\Delta$ has no missing faces of dimension $< k$, then this is additionally equivalent to
\begin{compactenum}[\rm (1)]
\setcounter{enumi}{+4}
\item $I_\Delta$, the Stanley--Reisner ideal of $\Delta$, admits a linear resolution.
\item  $\Delta^\ast$, the combinatorial Alexander dual of $\Delta$, is Cohen--Macaulay.
\end{compactenum}
\end{thm}

\begin{proof}
The equivalence $(2) \Longleftrightarrow  (3)$ follows from Fact \ref{fact:Res-Dec-Chordality}(4). The equivalence $(3)\Longleftrightarrow (4)$ was settled in \cite{KM} and (under the additional assumptions) $(4)\Longleftrightarrow (5) \Longleftrightarrow (6)$ are part of a classical characterization, cf. \cite{Froberg88}, \cite{Eagon-Reiner} with the last part of the equivalence known as the Eagon--Reiner Theorem. It remains to show that $(1)$ implies $(2)$ (the converse implication is trivial).

Let $\gamma$ denote any $j$-cycle of $\Delta$, $j\ge 2k$, and assume by induction that $\Delta$ is known to be resolution $i$-chordal for all $k\le i< j$, and that every $j$-cycle supported on less vertices then $\gamma$ admits a resolution. (The base case for the last assumption is that the simplex is resolution $l$-chordal for any $l$, see Fact~\ref{fact:Res-Dec-Chordality}(5).)

Define $\mu(\gamma)=(M,m)$ to denote the pair with $M$ equal to the maximum dimension of a missing face $\sigma$ supported in $\gamma$ (namely, $\sigma$ is a missing face of $\Delta$ all of whose facets are supported in $\gamma$), and $m>0$ equal to the number of missing $M$-faces supported in $\gamma$. Using the lexicographic order on $\N^2$, and adding a minimum $\hat{0}$ for the $(j+1)$-simplex,
we may assume by induction that every $j$-cycle $\gamma'$ with $\mu(\gamma')<_{\rm{lex}}\mu(\gamma)$ admits a resolution.

Now,
let $\sigma$ be any
minimal nonface supported in $\gamma$ of maximal dimension, let $v$ denote any of its vertices and let $\tau=\sigma \setminus \{v\}$.
Since $\Delta$ is resolution $(j-|\tau|)$-chordal, the cycle
$\Lk_\tau \gamma$
(defined by iteratively applying the coning isomorphism of Lemma~\ref{lem:cone} to $\gamma_{|(\St_\tau \Delta, \partial\St_\tau\Delta)}$) admits a resolution $r$.

Consider $\gamma':= \gamma - \gamma_{|(\St_\tau \Delta, \partial \St_\tau\Delta)} +(-1)^{j-1} \partial\tau * r$; it is a $j$-cycle. We shall see that either $\gamma'=0$ or $\mu(\gamma')<\mu(\gamma)$.

In the first case ($\gamma'=0$), $\gamma=\gamma_{|\St_\tau \Delta} +(-1)^j \partial\tau * r = \tau * r' +(-1)^j \partial\tau * r$ for some chain $r'$. Applying the differential to both sides we get $r'=\partial r$, hence $\gamma=\partial(\tau * r)$, and so $\tau*r$ resolves $\gamma$.
In the second case ($\gamma'\neq 0$), note that $\sigma$ is not supported on $\gamma'$, and that the missing faces of $\gamma'$ and not of $\gamma$ all have dimension $<\dim(\sigma)$, so $\mu(\gamma')<\mu(\gamma)$ and hence $\gamma'$ has a resolution, denote it by $s$.

Now, for any $(j-|\tau|)$-face $\sigma'$ in the support of $r$, we have $\tau *\sigma' \in \Delta$ as $\Delta$ has no missing faces of dimension $>k$, so $\tau * r$ is a chain of $\Delta$.
The sum $c:=s+(-1)^j\tau *r$ gives the desired resolution of $\gamma$: indeed,
$\partial c= \partial s + (-1)^{2j}\tau *\partial r +(-1)^j \partial \tau *r = \gamma'+\gamma_{|\St_v \Delta}+(-1)^j\partial\tau * r=\gamma$.
\end{proof}

\begin{rem} Notice that theorem \ref{thm:prp_exterior} is a strengthening of theorem \ref{mthm:persist}. A complex $\Delta$ with no missing faces of dimension bigger than $k$ satisfies $\Delta = \Cl_k(\Delta)$ thus resolution and decomposition $\ell$-chordality are equivalent for $\ell \ge k$ in this setting.
\end{rem}

This theorem is tight:
\begin{prp}\label{prop:tightPropagation}
For every $k\ge 2$, there is a simplicial complex $\mathfrak{J}_k$ that is resolution $i$-chordal for $i\in [k,2k-2]$, and has no missing face of dimension $>k$, but that is not $k$-Leray.
\end{prp}

\begin{proof}
Let $\Delta, \Delta'$ be two disjoint copies of the boundary complex of the $k$-simplex. Then the join
$\mathfrak{J}_k=\Delta \ast \Delta'$ has missing faces exclusively in dimension $k$, and it satisfies the desired chordality property since every induced subcomplex has homology only in dimension $k-1$ (precisely if this subcomplex coincides with $\Delta$ or $\Delta'$) or $2k-1$ (if it is the entire complex~$\mathfrak{J}_k$).
\end{proof}

\begin{rem}
The implication $(1)\Rightarrow (2)$ in Theorem~\ref{thm:prp_exterior} follows also from the following recent algebraic result of Herzog and Srinivasan \cite{HS}: let $\beta_{a,j}(\K)$ denote the algebraic Betti numbers of the Stanley-Reisner ring of $\K$, where $a,j\in \N$, and define
\[t_a(\K):= \max\{j:\ \beta_{a,j}(\K)\neq 0\} = \max\{j:\ \exists W\subseteq\K^{(0)}, |W|=j, \tilde{H}_{j-a-1}(\K_{|W};\mathbb{F})\neq 0\},\]
where the second equality holds by Hochster's formula, and $\mathbb{F}$ is the field fixed in Theorem~\ref{thm:prp_exterior}. It is shown in \cite[Cor. 4]{HS} that for any $a\geq 1$,
$t_a(\K)\leq t_{a-1}(\K)+t_1(\K)$.
\end{rem}
\begin{proof}[Proof of $(1)\Rightarrow (2)$:]
What we need to show is that for any $a\ge 1$, $t_a(\K)\le a+k$.
For $a=1$ this is obvious, as $t_1(\K)$ is the size of a maximal missing face of
$\K$, and by assumption $\K$ has no missing faces of $\dim>k$.
For $a>1$, by induction on $a$, and the result \cite[Cor.4]{HS} recalled above, $t_a(\K)\le t_{a-1}(\K)+t_1(\K)\leq (a-1)+k+(1+k)=a+2k$.
However, by the assumption on vanishing homology in (1), $t_a(\K)\notin [a+k+1, a+2k]$, so $t_a(\K)\le a+k$.
\end{proof}

\subsubsection*{Axiomatics}
We now notice that our theory remains valid if we consider slightly different homology theories. Consider any functor ${\mathcal F}$ from the category of simplicial complexes to the category of chain complexes over a ring $R$ such that the $k$-th module of ${\mathcal F}(\K)_\bullet$ is $C_k(\Delta, R)$, perhaps with a different differential.  A simplicial homology theory is any homology theory arising in this way. We say that such a theory is ``resolving'' if it satisfies the extended cone lemma, i.e.\ if $\lk_{\tau}\ :\ C_k(\K) \longrightarrow C_{k-1}(\eLk_{\tau} \K)$ induces a surjection in homology for any face $\tau\in\K$.
It is easy to check that a simplicial homology theory is resolving if
and only if it satisfies Theorem \ref{mthm:persist}.
\begin{ex}
The homology theory that takes $\Delta$ to  $H_\ast (\mc{M}\Delta)$, where $\mc{M}\Delta$ is the moment-angle complex of $\Delta$, cf. \cite{DJ}, is resolving.
\end{ex}
In particular, it is not necessary for the homology theory to satisfy invariance under homotopy equivalence in order to be resolving.

\section{Dirac's Theorem: Minimal homology cuts and Reverse Propagation}
\label{sec:Dirac}
In this section, we examine one of the most fundamental characterizations of
chordal graphs, the theorem of Dirac \cite{Dirac}, and observe that it is the basic case $k=1$ of what we call a \Defn{reverse propagation} phenomenon: i.e., a result
that asserts that decomposition $k$-chordal complexes contain nontrivial
decomposition $(k-1)$-chordal complexes.

\subsection{The classical case: Dirac's Theorems for graphs}

Dirac's Theorems are perhaps the most influential contributions to the theory of chordal graphs; they find applications, for instance, to colorings and computational graph theory. Recall that a \Defn{cut relative to vertices $v,w$} of a graph $G=(V,E)$, is a set
of vertices $V'\subset V\setminus \{v,w\}$ such that every path from $v$ to $w$
must contain a vertex of $V'$, i.e. $V'$ separates $v$ from $w$. The cut is
\Defn{minimal} if no proper subset of $V'$ separates $v$ from $w$. Dirac's first theorem characterizes relative minimal cuts:

\begin{thm}[Dirac's Minimal Cut Theorem]
In every chordal graph, the induced subgraph on a relative minimal cut is a
complete graph.
\end{thm}

By iteratively applying the Minimal Cut Theorem, until only one vertex is left in a connected component,
one immediately derives Dirac's Theorem on elimination orders in a chordal
graph.

\begin{thm}[Dirac's Reverse Propagation Theorem]
In every chordal graph, there is a vertex such that the induced subgraph on its link is complete.
\end{thm}

To justify the name we gave this theorem, notice that a complete graph is merely a resolution $0$-chordal complex. This motivates us to search for
higher-dimensional reverse propagation; more accurately, we are lead to
consider the following question, to be answered in the next subsection:
\begin{center}
\emph{In every decomposition $k$-chordal complex, must there be a face whose
extended link induces a subcomplex which is resolution $(k-1)$-chordal?}
\end{center}

\subsection{Dirac's Theorem in higher dimensions}
Before we address this problem, let us introduce another notion of higher
chordality:
\subsubsection*{Dirac complexes} Let us call a simplicial complex $\K$
\Defn{$k$-Dirac} if one of the following conditions holds:
\begin{enumerate}
\item $\Sk_k \K = \Sk_k \tau$ for some simplex $\tau$, or
\item There exists a face $\sigma$ of $\K$, $\dim \sigma \le k-1$, such that $\K-\sigma$ is $k$-Dirac and the extended link
$\eLk_\sigma \K$ is $(k-1)$-Dirac and satisfies
$\Cl_{k-1}\eLk_\sigma \K = \eLk_\sigma(\Cl_{k} \K)$.
\end{enumerate}
Note that a graph is chordal if and only if it is $1$-Dirac.
Compared to resolution and decomposition chordality, the Dirac property is combinatorial rather than homological, and one can
therefore expect it to be more restrictive than these homological notions:

\begin{prp}\label{prp:dirac_to_chordal}
Every $k$-Dirac complex is
decomposition $k$-chordal, for every ring of coefficients $R$.
\end{prp}

\begin{proof}
Let $\K$ be $k$-Dirac. If $\K$ has a complete $k$-skeleton it is clearly $\cdd_k$-chordal; else let $\sigma\in\K^{(\leq k-1)}$ be such that $\K-\sigma$ is $k$-Dirac and $\eLk_\sigma \K$ is $(k-1)$-Dirac with
$\Cl_{k-1}\eLk_\sigma \K = \eLk_\sigma(\Cl_{k} \K)$.
By Lemma~\ref{fact:Res-Dec-Chordality}(2), we need to show that $\Cl_k\K$ is $\cdr_k$-chordal. We proceed by double induction, on dimension and number of faces.


Consider a $k$-cycle $z$ in $\Cl_{k} \K$, $z=z_0+z_1$ where $z_0$ is the restriction of $z$ to faces containing $\sigma$, and $z_1=z-z_0$.
Thus, $z_0$ is a (relative) cycle in $(\St_\sigma(\Cl_k \K), \eLk_\sigma(\Cl_k \K))$. As $\eLk_\sigma \K$ is $(k-1)$-Dirac, and hence by induction $\Cl_{k-1}\eLk_\sigma \K$ is $\cdr_{k-1}$-chordal, by the Extended Cone
Lemma~\ref{lem:econe}, and the requirement
$\Cl_{k-1}\eLk_\sigma \K = \eLk_\sigma(\Cl_{k} \K)$,
there is a (relative) $k$-chain $c_0$ that resolves $z_0$; that is, $c_0^{(0)}=z_0^{(0)}$ and $\partial c_0 = z_0 + b$ where $b$ is a $k$-chain in $\eLk_\sigma \Cl_k \K$, hence also in $\Cl_k(\K-\sigma)$. As both $z_0+b$ and $z_0+z_1$ are $k$-cycles, we get that $z_1-b$ is a $k$-cycle in $\Cl_k(\K-\sigma)$.
As $\K-\sigma$ is $k$-Dirac, by induction $z_1-b$ has a resolution $c_1$.
Thus, $c=c_0+c_1$ resolves $z$.
\end{proof}

The converse to this proposition does not hold:

\begin{ex}\label{ex:chordal_not_dirac}
Recall that the dunce hat is the quotient topological space obtained from identifying the edges of a triangle with an orientation that is not cyclic; see e.g. \cite{Zeeman-Dunce} for more on this remarkable space. Let $D$ denote any \emph{flag} triangulation of the Dunce hat, e.g. the barycentric subdivision of some triangulation of it; minimal 8 vertex triangulations are known, e.g. \cite{Benedetti-Lutz-Dunce}
(or any other contractible $2$-complex without free edges, i.e. where every edge is contained in at least two facets, will do).
Since $D$ is contractible it does not support a $2$-cycle; hence, $D$
is decomposition $2$-chordal. However, the extended links of all edges and vertices contain induced
circles of length at least $4$, and hence are not $1$-Dirac.
\end{ex}

\begin{rem}\label{rem:Woodroofe}
The notion of $k$-Dirac is more general than the combinatorial chordality notions of \cite{Woodroofe} and \cite{Emtander}, and therefore more general than the one of \cite{Ha-VanT} as well. All these notions imply the existence of a linear resolution and are incomparable as explained by Woodroofe \cite[Example 4.8]{Woodroofe}. The approach of both is similar: they provide combinatorial criteria on the hypergraph/clutter of minimal non-faces of $\Delta$ and use these criteria to show the existence of a linear resolution. Both criteria are inductive in nature.
In fact, choosing $\sigma$ as in the definition of $k$-Dirac to be a vertex is always possible for such complexes.
For an example of $\K$ which is $k$-Dirac for any $k\geq d$, and
its missing faces form a $(d+1)$-uniform clutter whose complement $(d+1)$-clutter is not chordal by any of \cite{Woodroofe} and \cite{Emtander}, take $\Delta$ (for $d=2$) to be the join of an edge with a square boundary, union with the two diagonals of the square. Indeed, one can verify that $\K$ is $k$-Dirac for any $k\geq 2$ (again, choosing $\sigma$ in the definition to be a vertex is always possible), but no vertex of $\K$ is \emph{simplicial} in the sense of \cite{Woodroofe}, and similarly for \cite{Emtander}.

In \cite{Cord-Lem-Sa} another combinatorial definition of elimination order was introduced, and was shown to imply decomposition $k$-chordality \cite[Thm.5.2]{Cord-Lem-Sa}. Compared to our definition of $k$-Dirac, their face $\sigma$ is always $(k-1)$-dimensional and $\eLk_\sigma \K$ is always the join of $\partial \sigma$ with a nonempty simplex, and in particular satisfies the conditions in our definition. Further, their complex must be the $k$-clique complex of a pure $k$-dimensional complex, thus the complex $\Delta$ above does not satisfy the definition in \cite{Cord-Lem-Sa} either.

\end{rem}
\subsubsection*{Cuts and homology cuts}
We need an analogue of cuts for higher dimensional complexes; this is straightforward:
denote by $\Gamma_k \K$ the graph whose vertices are the $k$-faces of $\K$, and two are connected by an edge if their intersection is a $(k-1)$-face of $\K$.
A $(k-1)$-dimensional subcomplex $\mc{C}$ of a simplicial complex $\K$ is a \Defn{$k$-cut relative to two faces $\sigma$ and $\tau$} if every path in $\Gamma_k \K$ from a
$k$-face containing $\sigma$ to a $k$-face containing $\tau$ must
pass through one of the $(k-1)$-faces of $\mc{C}$.
When $\sigma$ and $\tau$ are understood, we say that $\mc{C}$ is a relative $k$-cut, or just a $k$-cut, for short.
The \Defn{component}
$\K^{\mc{C}}_\sigma$ of $\sigma$ in $\K$ w.r.t.\ $\mc{C}$ is the
subcomplex of $\K$ induced by $k$-faces that are connected to $\St_\sigma
\K$ by a path in $\Gamma_k \K$ that does not
intersect $\mc{C}$; the component $\K^{\mc{C}}_\tau$ is defined
analogously. A $k$-cut is \Defn{minimal} if it is inclusion minimal.

We say that
a relative $k$-cut
 $\mc{C}$ is a \Defn{homology $k$-cut} if for every $(k-1)$-cycle $z$ in $\mc{C}$, and for
$L=\K^{\mc{C}}_\sigma$ or $\K^{\mc{C}}_\tau$ (or both, in which case we say the homology $k$-cut is \Defn{two-sided}),
there is a $k$-chain $c$ in $L$ such that
\[ c^{(0)} \cap \mc{C} = z^{(0)} \quad \text{and}\quad \partial c =z.\]

The next subsection clarifies the hierarchy between these definitions.

\subsubsection*{Existence of minimal cuts and homology cuts}
First, we note that minimal $k$-cuts always exist:

\begin{prp}
Let $\K$ denote any simplicial complex that has at least two $k$ faces.
 Then $\K$ has a minimal $k$-cut (relative to any two $k$-faces $\sigma$ and $\tau$ of $\K$).
\end{prp}

\begin{proof}
Clearly, $\K$ admits a cut, e.g. $\mc{C}=\partial\sigma$; passing to an inclusion minimal subcut yields the claim.
\end{proof}

We can guarantee minimal cuts of a special form:
\begin{thm}\label{cor:existence_minimal_cut}
Let $\K$ denote any $k$-dimensional simplicial complex that has at least two $k$ faces. Assume $\Gamma_k\K$ is connected.
Then $\K$ has a face
$\sigma$ of dimension at most $k-1$
whose extended link is a minimal $k$-cut.
\end{thm}
\begin{proof}
By assumption, $\K$ is not the $k$-simplex, hence for some $(k-1)$-face $\tau$, $\eLk_{\tau}$ is a $k$-cut. Assume $\eLk_{\tau}$ is not minimal, otherwise we are done. Then there is a $(k-1)$-face $\tau'$ of $\eLk_{\tau}$ such that $\St_{\tau'}\subseteq \St_{\tau}$. There is exactly one $k$-face $\sigma'$ which contains both $\tau$ and $\tau'$ and thus $\St_{\tau'}=2^{\sigma'}$, and $\eLk_{\tau'}$ is also a $k$-cut. Let $\sigma$ be the intersection of all facets of $\sigma'$ whose (closed) star equals $\sigma'$. (This intersection is nonempty as $\Gamma_k\K$ is connected.)
Then $\eLk_{\sigma}$ is a minimal $k$-cut.
\end{proof}
The cuts in the above theorem are also homology cuts:
\begin{cor}\label{cor:existence_Homology_minimal_cut}
Under the assumptions of Theorem \ref{cor:existence_minimal_cut}, the minimal $k$-cut $\eLk_{\sigma}\K$ is a homology cut.
\end{cor}
\begin{proof}
The component $L=\St_{\sigma}$ is acyclic, hence any $(k-1)$-cycle $z$ in $\eLk_{\sigma}$ is the boundary of some $k$-chain $c$ in $L$.
In fact, $z$ must be of the form $z=\partial\sigma * z'$ for some cycle $z'$, hence $c$ can be chosen to be $c=\sigma * z'$. Thus,
for $k>1$, $c$ and $z$ have the same vertex support. For $k=1$, Dirac's Minimal Cut theorem implies the claim.
\end{proof}
However, two-sided homology cuts are harder to guarantee, and not every homology cut is two-sided:

\begin{ex}[No two-sided homology cuts]\label{ex:no_twosided_homology_cut}
Consider a triangulation of the dunce hat $D$; see also Example \ref{ex:chordal_not_dirac}. The extended link of each face $\sigma$ is \emph{not} a two-sided homology $2$-cut, else, a $1$-cycle $z$ in $\eLk_{\sigma}$ (which exists!)  would be the boundary of two $2$-chains $s,t$, one in each component, and thus $D$ would contain the $2$-cycle $s-t$, contradicting that the dunce hat is acyclic.
\end{ex}

\subsubsection*{Reverse Propagation}
Nevertheless, when two-sided homology $k$-cuts exist and a few additional mild conditions are satisfied, reverse propagation holds, which we discuss next.
\begin{thm}\label{thm:two-sided_reverse_propagation}
Let $\K$ be a decomposition $k$-chordal  simplicial complex, and let $\sigma\in\K$ be a $k$-face.
Assume $\eLk_{\sigma}\K^{(k-1)}$ is a two-sided homology $k$-cut for $\K^{(k)}$ (relative to two $k$-faces), and that $\eLk_{\sigma}\K^{(k)}$ is a $(k+1)$-cut for $\K^{(k+1)}$.
Then, $\eLk_{\sigma}\K$ is decomposition $(k-1)$-chordal.
\end{thm}
\begin{proof}
The proof follows the ideas of Dirac: Let $z$ denote any $(k-1)$-cycle in
$\eLk_{\sigma}$.
By the definition of two-sided homology cuts,
there exist $k$-chains $s$ and $t$, one in each component defined by the $k$-cut
$L=\eLk_{\sigma}^{(k-1)}$, such that
$\partial s=\partial t=z$ and
$s^{(0)}\cap L^{(0)} = t^{(0)}\cap L^{(0)} = z^{(0)}.$

Then the chain $\wh{z}=s-t$ is a $k$-cycle, so by $\updelta_k$-chordality of $\K$ it has a resolution $\wh{c}$.
Since $L'=\eLk_{\sigma}^{(k)}$ is a $(k+1)$-cut,
it separates $\wh{c}$ into chains $\wh{s}$ and $\wh{t}$, with $\wh{c}=\wh{s}-\wh{t}$. Then $\partial\wh{s}$ and $s$ are identical on the restriction to $k$-faces not in $L'$; similarly for $\partial\wh{t}$ and $t$. Thus, $s-\partial\wh{s}=-\partial \wh{t}+t$ forms the desired resolution of $z$.
\end{proof}
We end with the following problem:
\begin{problem}
For any $k>1$, find interesting families of complexes $\K$ for which there exists a face $\sigma$ satisfying the conditions in Theorem \ref{thm:two-sided_reverse_propagation}.
\end{problem}

\section*{Acknowledgments}
We thank Satoshi Murai and the referee for pointing to us the relevance of \cite{HS} and \cite{Cord-Lem-Sa} respectively,
and Isabella Novik and the referee for helpful remarks on the presentation.
The third author also thanks Isabella Novik for the research assistant positions funded through NSF Grant DMS-
1069298 and NSF Grant DMS-1361423.
{\small
\bibliographystyle{myamsalpha}
\bibliography{rigidity}
}
\end{document}